\documentclass[12pt]{amsart}

\usepackage[pdfauthor   = {Mohamed\ Barakat\ and\ Reimer\ Behrends\ and\ Christopher\ Jefferson\ and\ Lukas\ Kuehne\ and\ Martin\ Leuner},
            pdftitle    = {On\ the\ generation\ of\ rank\ 3\ simple\ matroids},
            pdfsubject  = {},
            pdfkeywords = {rank 3 simple matroids;\ integrally splitting characteristic polynomial;\ Terao's freeness conjecture;\ recursive iterator;\ tree-iterator;\ iterators of leaves of rooted tree;\ priority queue\; parallel evaluation of recursive iterator;\ leaf-iterator;\ noSQL database;\ ArangoDB},
            bookmarks=true,
            bookmarksopen=true,
            pagebackref=true,
            hyperindex=true,
            colorlinks=true,
            linkcolor=blue,
            citecolor=blue,
            filecolor=blue,
            urlcolor=blue,
            ]{hyperref}

\usepackage[utf8]{inputenc}
\usepackage[english]{babel}
\usepackage[T1]{fontenc}
\usepackage{geometry}                
\usepackage{times}
\usepackage{mathrsfs}
\usepackage{mathtools}
\usepackage{latexsym}
\usepackage{amssymb}
\usepackage{amsthm}
\usepackage{epsfig}
\usepackage{colortbl}
\usepackage[all]{xy}
\usepackage{fancyvrb}
\usepackage{graphicx}
\usepackage[font=small,labelfont=bf]{caption}
\usepackage[dvipsnames]{xcolor}
\usepackage{accents} 
\usepackage{enumerate}

\usepackage{wrapfig}
\usepackage{tikz}
\usetikzlibrary{shapes,arrows,matrix,backgrounds,positioning,plotmarks,calc,patterns,
decorations.shapes,
decorations.fractals,
decorations.markings,
decorations.pathreplacing,
decorations.pathmorphing,
decorations.text
}
\usepackage[colorinlistoftodos,shadow]{todonotes}

\usepackage{bm}
\usepackage{multirow}
\usepackage{mdwlist}
\usepackage{stmaryrd}
\usepackage{mathdots} 

\usepackage[toc,page]{appendix}
\usepackage{float}

\usepackage{bigfoot}

\usepackage[sort&compress,capitalise]{cleveref}

\crefname{exmp}{Example}{Examples}

\usepackage[linesnumbered,commentsnumbered,ruled,vlined]{algorithm2e}

\SetCommentSty{mycommfont}

\newtheoremstyle{mytheoremstyle} 
    {5pt}                    
    {5pt}                    
    {\itshape}                   
    {\parindent}                           
    {\bf}                   
    {.}                          
    {.5em}                       
    {}  

\theoremstyle{mytheoremstyle}

\newtheorem{theorem}{Theorem}[section]

\newtheorem{prop}[theorem]{Proposition}

\newtheoremstyle{mytdefintionstyle} 
    {5pt}                    
    {5pt}                    
    {\rm}                   
    {\parindent}                           
    {\bf}                   
    {.}                          
    {.5em}                       
    {}  

\theoremstyle{remark}
\newtheorem{rmrk}[theorem]{Remark}

\theoremstyle{mytdefintionstyle}
\newtheorem{defn}[theorem]{Definition}
\newtheorem{exmp}[theorem]{Example}

\newtheorem*{conv}{Convention}

\newtheoremstyle{exmp_contd}
    {5pt}                    
    {5pt}                    
    {\rm}                   
    {\parindent}                           
    {\bf}                   
    {.}                          
    {.5em}                       
    {\thmname{#1}\ \thmnumber{ #2}\thmnote{#3}\ (continued)}  
\theoremstyle{exmp_contd}

\newcommand{\GAP}{\textsf{GAP}\xspace}
\newcommand{\HPCGAP}{\textsf{HPC-GAP}\xspace}

\DeclareMathOperator{\Spec}{Spec}

\renewcommand\AA{\mathbb{A}}

\newcommand\A{\mathcal{A}}

\newcommand\F{\mathcal{F}}

\newcommand\bF{\mathbb{F}}

\newcommand{\Q}{\mathbb{Q}}
\newcommand{\CC}{\mathbb{C}}

\newcommand{\Z}{\mathbb{Z}}
\newcommand\N{\mathbb{N}}
\renewcommand\phi{\varphi}

\DeclareMathOperator\Sym{Sym}

\DeclareMathOperator{\Stab}{Stab}

\definecolor{darkgray}{rgb}{0.3,0.3,0.3}
\definecolor{LightGray}{gray}{0.9}

\pgfarrowsdeclarecombine[\pgflinewidth]
  {doublestealth}{doublestealth}{stealth'}{stealth'}{stealth'}{stealth'}

\definecolor{darkgreen}{rgb}{0.008,0.617,0.067}
\definecolor{brown}{rgb}{0.6,0.4,0.2}

\newif\ifjournalversion

\author[M. Barakat]{Mohamed Barakat}
\address{Department of mathematics, University of Siegen, 57068 Siegen, Germany}
\email{\href{mailto:Mohamed Barakat <mohamed.barakat@uni-siegen.de>}{mohamed.barakat@uni-siegen.de}}

\author[R. Behrends]{Reimer Behrends}
\address{Department of mathematics, University of Kaiserslautern, 67653 Kaiserslautern, Germany}
\email{\href{mailto:Reimer Behrends <behrends@gmail.com>}{behrends@gmail.com}}

\author[C. Jefferson]{Christopher Jefferson}
\address{School of Computer Science, University of St Andrews, KY16 9SX St Andrews, United Kingdom }
\email{\href{mailto:Christopher Jefferson <caj21@st-andrews.ac.uk>}{caj21@st-andrews.ac.uk}}

\author[L. K\"uhne]{Lukas K\"uhne}
\address{Einstein Institute of Mathematics, The Hebrew University of Jerusalem, Giv’at Ram, Jerusalem, 91904, Israel}
\address{Max Planck Institute for Mathematics in the Sciences, Inselstr. 22, 04103, Leipzig, Germany}
\email{\href{mailto:Lukas Kuehne<lukas.kuhne@mis.mpg.de>}{lukas.kuhne@mis.mpg.de}}

\author[M. Leuner]{Martin Leuner}
\address{Lehrstuhl B für Mathematik, RWTH Aachen University, Germany}
\email{\href{mailto:Martin Leune <martin.leuner@rwth-aachen.de>}{martin.leuner@rwth-aachen.de}}

\begin{document}

\title[On the generation of rank $3$ simple matroids]{On the generation of rank $3$ simple matroids \\ with an application to Terao's freeness conjecture}
\begin{abstract}
  In this paper we describe a parallel algorithm for generating all non-isomorphic rank $3$ simple matroids with a given multiplicity vector.
  We apply our implementation in the HPC version of \GAP to generate all rank $3$ simple matroids with at most $14$ atoms and an integrally splitting characteristic polynomial.
  We have stored the resulting matroids alongside with various useful invariants in a publicly available, ArangoDB-powered database.
  As a byproduct we show that the smallest divisionally free rank $3$ arrangement which is not inductively free has $14$ hyperplanes and exists in all characteristics distinct from $2$ and $5$.
  Another database query proves that Terao's freeness conjecture is true for rank $3$ arrangements with $14$ hyperplanes in any characteristic.
\end{abstract}

\thanks{This work is a contribution to Project II.1 of SFB-TRR 195 'Symbolic Tools in Mathematics and their Application' funded by Deutsche Forschungsgemeinschaft (DFG).
The fourth author was supported by ERC StG 716424 - CASe, a Minerva Fellowship of the Max Planck Society and the Studienstiftung des deutschen Volkes.}

\keywords{%
rank $3$ simple matroids,
integrally splitting characteristic polynomial,
Terao's freeness conjecture,
recursive iterator,
tree-iterator,
leaf-iterator,
iterator of leaves of rooted tree,
priority queue,
parallel evaluation of recursive iterator,
noSQL database,
ArangoDB%
}
\subjclass[2010]{%
05B35,
52C35,
32S22,
68R05,
68W10%
}
\maketitle


\section{Introduction} \label{sec:Intro}

In computational mathematics one often encounters the problem of scanning (finite but) large sets of certain objects.
Here are two typical scenarios:
\begin{itemize}
  \item Searching for a counter-example of an open conjecture among these objects.
  \item Building a database of such objects with some of their invariants.
\end{itemize}
A database is particularly useful when the questions asked are relational, i.e., involve more than one object (cf.~\Cref{rmrk:ind_div_using_db}).
Recognized patterns and questions which a database answers affirmatively may lead to working hypotheses or even proofs by inspection (cf.~\Cref{thm:terao}).

In any such scenario there is no need to simultaneously hold the entire set in RAM.
It is hence important to quickly \emph{iterate} over such sets in a memory efficient way rather than to enumerate them.

The central idea is to represent each such set $T$ as the set of leaves of a rooted tree $T_\bullet$ (cf.~\Cref{sec:trees}).
In other words, we embed $T$ as the set of leaves in the bigger set of vertices $V(T_\bullet)$.
We then say that $T_\bullet$ \textbf{classifies} $T$.
The internal vertices of the tree $T_\bullet$ are usually of different nature than the elements of $T$.
Their sole purpose is to encode common pre-stages of the leaves.
To iterate over the vertices of the rooted tree $T_\bullet$ we introduce the data structure of a tree-iterator $t$ (cf.~\Cref{defn:recursive_iterator}).

In this article we will describe how to use tree-iterators to classify all nonisomorphic simple rank $3$ matroids with up to $14$ atoms and integrally splitting characteristic polynomial.

A simple matroid $M$ of rank $3$ on $n$ labeled points corresponds to a bipartite graph $G_M$ (cf.~\Cref{rmrk:monoidal_partition}).
We denote by $(m_2, \ldots, m_{n-1})$ the \textbf{multiplicity vector} of $M$ where $m_k$ is the number of coatoms of multiplicity $k$, i.e., the degree in the bipartite graph corresponding to $M$ (cf.~\Cref{def:mult_vector}).
The multiplicity vector determines the characteristic polynomial of $M$:
\begin{equation} \label{eq:chiM} \tag{*}
  \frac{\chi_M(t)}{t-1} = t^2 - (n-1) t + (b_2-(n-1)) \quad \mbox{with} \quad b_2 \coloneqq \sum_{k=2}^{n-1} m_k(k-1) \mbox{.}
\end{equation}
In fact, two simple rank $3$ matroids (or more generally, two paving matroids) have the same multiplicity vector $(m_2, \ldots, m_{n-1})$ iff their Tutte polynomials coincide \cite{Bryl72}.

After extending the notions of inductive and divisional freeness from arrangements to matroids (see Definitions~\ref{def:if} and~\ref{def:df}) we get the following table of cardinalities\footnote{Apart from the number of simple matroids, we were unable to find any of the sequences in the above table in the OEIS database.} of certain classes of nonisomorphic simple rank $3$ matroids.
A matroid is called \emph{Tutte-unique} or \emph{T-unique} if it is determined up to isomorphism by its Tutte polynomial\footnote{The Tutte polynomial of all rank $3$ matroids with an integrally splitting characteristic polynomial and up to $13$ atoms was computed using the $\mathtt{GAP}$ package $\mathtt{alcove}$ \cite{alcove}.} (see~\cite{MN05} for a survey on T-unique matroids).
The content of the table can be reconstructed using the database \cite{matroids_split}.

{\scriptsize
\begin{figure}[H]
\addtolength{\tabcolsep}{0.9pt}
\begin{tabular}{r|rrrrrrrrrrrr}
number of atoms & 3 &4 & 5 & 6 & 7 & 8 & 9 & 10 & 11 & 12 & 13 & 14 \\
\hline
\multicolumn{1}{c}{} & \multicolumn{12}{c}{rank $3$, simple matroids} \\
\hline
\rowcolor{LightGray}
simple matroids & 1 & 2 & 4 & 9 & 23 & 68 & 383 & 5 249 & 232 928 & 28 872 972 & ? & ? \\
\hline
\hline
integrally splitting $\chi_M(t)$ & 1 & 1 & 2 & 3 & 7 & 7 & 17 & 35 & 163 & 867 & 30 724 & 783 280 \\
\hline
\rowcolor{LightGray}
divisionally free & 1 & 1 & 2 & 3 & 6 & 7 & 15 & 33 & 147 & 857 & 28 287 & 781 795 \\
\hline
inductively free & 1 & 1 & 2 & 3 & 6 & 7 & 15 & 33 & 147 & 839 & 27 931 & 750 305\\
\hline
\rowcolor{LightGray}
supersolvable & 1 & 1 & 2 & 3 & 5 & 7 & 11 & 20 & 41 & 118 & 518 & 4 820 \\
\hline
\multicolumn{1}{c}{} & \multicolumn{12}{c}{representable, rank $3$, simple matroids} \\
\hline
rep. \& int.\ split.\ $\chi_M(t)$ & 1 & 1 & 2 & 3 & 7 & 7 & 17 & 30 & 86 & 208 & 999 & 1 574 \\
\hline
\rowcolor{LightGray}
rep. \& divisionally free & 1 & 1 & 2 & 3 & 6 & 7 & 15 & 28 & 75 & 198 & 631 & 1 401 \\
\hline
rep. \& inductively free & 1 & 1 & 2 & 3 & 6 & 7 & 15 & 28 & 75 & 198 & 631  & 1 400\\
\hline
\rowcolor{LightGray}
rep. \& supersolvable & 1 & 1 & 2 & 3 & 5 & 7 & 11 & 20 & 35 & 82 & 223 & 649\\
\hline
\multicolumn{1}{c}{} & \multicolumn{12}{c}{Tutte-unique, rank $3$, simple matroids} \\
\hline
T.-u.\  \& int.\ split.\ $\chi_M(t)$ & 1 & 1 & 2 & 3 & 7 & 5 & 11 & 10 & 17 & 17 & 18 & 23 \\
\hline
\rowcolor{LightGray}
T.-u.\  \& divisionally free & 1 & 1 & 2 & 3 & 6 & 5 & 9 & 10 & 14 & 16 & 17 & 21 \\
\hline
T.-u.\  \& inductively free & 1 & 1 & 2 & 3 & 6 & 5 & 9 & 10 & 14 & 16 & 17  & 21 \\
\hline
\rowcolor{LightGray}
T.-u.\  \& supersolvable & 1 & 1 & 2 & 3 & 5 & 5 & 8 & 10 & 12 & 14 & 15 & 19 \\
\hline
\multicolumn{1}{c}{} & \multicolumn{12}{c}{representable, Tutte-unique, rank $3$, simple matroids} \\
\hline
rep.\  \& T.-u.\  \& int.\ split.\ $\chi_M(t)$ & 1 & 1 & 2 & 3 & 7 & 5 & 11 & 10 & 16 & 17 & 17 & 22\\
\hline
\rowcolor{LightGray}
rep.\  \& T.-u.\ \&  div.\  free & 1 & 1 & 2 & 3 & 6 & 5 & 9 & 10 & 13 & 16 & 16 & 20 \\
\hline
rep.\  \& T.-u.\ \&  ind.\  free & 1 & 1 & 2 & 3 & 6 & 5 & 9 & 10 & 13 & 16 & 16 & 20 \\
\hline
\rowcolor{LightGray}
rep.\  \& T.-u.\ \& supersolvable & 1 & 1 & 2 & 3 & 5 & 5 & 8 & 10 & 12 & 14 & 15 & 19 \\
\hline
\end{tabular}
\captionof{table}{\rule{0em}{2em}Cardinalities of certain classes of nonisomorphic simple rank $3$ matroids.
\label{tbl:simple_matroids}}
\addtolength{\tabcolsep}{-4pt}
\end{figure}
}

The total number of simple rank $3$ matroids with $n \leq 12$ (unlabeled) atoms\footnote{\url{http://oeis.org/A058731}} is taken from \cite{MMIB}.
This number also coincides with the number of linear geometries minus one with $n \leq 12 $ (unlabeled) points\footnote{\url{http://oeis.org/A001200}} and has been determined earlier in \cite{BB99}.


Using our algorithm in \textsf{HPC-GAP} we directly computed all $815 107$ simple rank $3$ matroids with integrally splitting characteristic polynomial with up to $n=14$ atoms and stored them in the database \cite{matroids_split}.
Subsequently, we verified our counting by comparing it against the matroids with integrally splitting characteristic polynomial for $n \leq 11$ in \cite{MMIB_DB}\footnote{We wrote a short program to compute the characteristic polynomial of these matroids as the matroids come without precomputed properties in~\cite{MMIB_DB}.}.

\subsection{Applications of the Database}

Inspecting the database \cite{matroids_split} enables us to investigate questions like:

\begin{enumerate}
  \item Is being divisionally or inductively free a property determined by the Tutte polynomial? \\
    We answer this question negatively in \Cref{exmp:if_pair}. \label{q:a}
  \item What is the smallest number of atoms of a representable rank $3$ matroid which is divisionally free but not inductively free?\footnote{
  It is already known that such a matroid exists, namely the rank $3$ reflection arrangement $\mathcal{A}(G_{24})$ (with $21$ hyperplanes) of the exceptional complex reflection group $W=G_{24}$ is recursively free \cite{Muc17} but not inductively free \cite{HR15}.
  Hence, an addition of $\mathcal{A}(G_{24})$ is easily seen to be divisionally free but not inductively free.
  Therefore, the sequences of representable divisionally free and inductively free matroids differ at $n=22$ at the latest.} \\
  We answer this question in \Cref{exmp:div_free_not_ind_free}.
  \item Does the database confirm Terao's freeness conjecture for further classes of arrangements? \\
    Indeed, this is \Cref{thm:terao}.
\end{enumerate}
Some of these questions require the construction of \emph{all} matroids with the corresponding number of atoms first, demonstrating the usefulness of a database.

\begin{exmp}\label{exmp:if_pair}
Consider the rank $3$ matroids $M_1$ and $M_2$ of size $11$ given below by the adjacency lists $A_1$ and $A_2$ of their corresponding bipartite graph respectively.
{\tiny
\begin{align*}
A_1 \coloneqq& \{\{1,2,3,4\},\{1,5,6,7\},\{1,8,9,10\},\{2,5,8,11\},\{3,6,9,11\},\{2,6,10\},\{2,7,9\},\{3,5,10\},\{4,5,9\},\{4,7,11\},\\
&\{1,11\},\{3,7\},\{3,8\},\{4,6\},\{4,8\},\{4,10\},\{6,8\},\{7,8\},\{7,10\},\{10,11\}\},\\
A_2 \coloneqq & \{\{1,2,3,4\},\{1,5,6,7\},\{2,5,8,9\},\{3,6,8,10\},\{4,7,9,10\},\{1,8,11\},\{2,7,11\},\{3,9,11\},\{4,6,11\},\{5,10,11\},\\
&\{1,9\},\{1,10\},\{2,6\},\{2,10\},\{3,5\},\{3,7\},\{4,5\},\{4,8\},\{6,9\},\{7,8\}\}.
\end{align*}
}
The matroids $M_1$ and $M_2$ are representable over $\Q$ and $\Q(\sqrt{5})$, respectively.
They admit the following representation matrices, respectively:
{\tiny
\begin{align*}
R_1 \coloneqq  &
\begin{pmatrix*}[r]
  1 & 0 & 1 & 1   & 0 & 1 & 1   & 0& 1  & 1& 0 \\
 0 & 1 & 1 & \frac{1}{2}   & 0 & 0 & 0   & 1& \frac{1}{2}  & 1& 1 \\
 0 & 0 & 0 & 0   & 1 & 1 & \frac{1}{2}   & 1& \frac{1}{2}  & 1& -1  
  \end{pmatrix*}\mbox{,}\\
  R_2\coloneqq  &
  \begin{pmatrix*}[r]
  1& 0& 1& 1& 0& 1& 1& 0& 0& 1& 1 \\
  0& 1& 1& \phi + 1& 0& 0& 0& 1& 1& -\phi& -\phi \\ 
  0& 0& 0& 0& 1& 1& \phi& -1& -\phi+1& \phi+1& \phi  
    \end{pmatrix*}\mbox{,}
\end{align*}
}
where $\phi=\frac{1+\sqrt{5}}{2}$ denotes the golden ratio.
Their multiplicity vectors agree and are given by $(m_k)=(m_2,m_3,m_4)=(10,5,5)$.
Hence, their Tutte polynomials also agree:
{\tiny
\[
T_{M_1}(x,y) = T_{M_2}(x,y)=y^8+3y^7+6y^6+10y^5+15y^4+x^3+5xy^2+21y^3+8x^2+15xy+23y^2+16x+16y\mbox{.}
\]
}
Both $M_1$ and $M_2$ have an integrally splitting characteristic polynomial:
\[
\chi_{M_1}(t)=\chi_{M_2}(t)=(t-1)(t-5)^2\mbox{.}
\]
Using the database we found that $M_1$ is inductively free and hence divisionally free whereas $M_2$ is not even divisionally free.
We checked with \GAP that any representation of $M_2$ is a free arrangement.
Both are not supersolvable.

The database also shows that for rank $3$ matroids this example is minimal with respect to the number of elements.

  Finally, the corresponding question \eqref{q:a} for the stricter notion of supersolvability is confirmed by the database and already proven for rank $3$ matroids in \cite[Proposition 4.2]{Abe17}.
  The proof is formulated for arrangements but works without changes for matroids.
\end{exmp}

\begin{exmp} \label{exmp:div_free_not_ind_free}
  By inspecting the database we found that among the rank $3$ matroids with up to $14$ atoms there is a unique representable matroid $M$ with 14 atoms which is divisionally free but not inductively free.
  It can be represented by the following matrix:
  \[
    \begin{pmatrix*}[r]
      1 &0 &1 &1 & 1 & 0 &1 &1 & 1 & 0 & 1 & 1 & 0 & 1 \\
      0 &1 &1 &2a-1 &2a &0 &0 &0 & 0 & 1 & -2a+2 &1 & 1 & 1 \\
      0 &0 &0 &0 & 0 & 1 &1 &-2a+1 &-a+1 &a &1 & a &2a-1 &1
    \end{pmatrix*}\mbox{,}
  \]
  where $a$ satisfies the equality $2 a^2 - 2 a + 1=0$ and the inequation $(3 a-1)(a+1)\neq 0$.
  In particular $M$ is representable in any characteristic distinct from $2$ and $5$ \cite{DivFreeNotIndFree}.
  
  Its characteristic polynomial is $\chi_M(t)=(t-1)(t-6)(t-7)$. The restriction $M''$ of $M$ to its third atom (resp. hyperplane) has characteristic polynomial $\chi_M(t)=(t-1)(t-6)$ which shows that any arrangement representing $M$ is divisionally free (cf. \cref{def:df}).
  Furthermore, the Tutte polynomial of $M$ is
  {\tiny
  \[
  y^{11}+3y^{10}+6y^9+10y^8+15y^7+21y^6+28y^5+2xy^3+36y^4+x^3+10xy^2+43y^3+11x^2+24xy+43y^2+30x+30y.
  \]
  }
\end{exmp}


A central notion in the study of hyperplane arrangements is freeness.
A central arrangement of hyperplanes $\A$ is called \emph{free} if the derivation module $D(\A)$ is a free module over the polynomial ring.
An important open question in this field is Terao's conjecture which asserts that the freeness of an arrangement over a field $k$ only depends on its underlying matroid and the characteristic of $k$.
It is known that Terao's conjecture holds for arrangements with up to $12$ hyperplanes in characteristic $0$ (cf. \cite{FV14,ACKN16}).
Recently, Dimca, Ibadula, and Macinic confirmed Terao's conjecture for arrangements in $\CC^3$ with up to 13 hyperplanes \cite{DIM19}.

Inspecting our database we obtain the following result:
\begin{theorem}\label{thm:terao}
	Terao's freeness conjecture is true for rank $3$ arrangements with $14$ hyperplanes in any characteristic.
\end{theorem}

This article is organized as follows:
In \Cref{sec:matroids} we recall the notion of a matroid and introduce several subclasses of simple rank $3$ matroids.
In \Cref{sec:matroid_iter} we discuss the Algorithm used to construct a tree-iterator classifying all nonisomorphic simple rank $3$ matroids with up to $n=14$ atoms having an integrally splitting characteristic polynomial.
In \Cref{sec:representability} we briefly point out how to use Gröbner bases to compute the moduli space of representations (over some unspecified field $\bF$) of a matroid as an affine variety over $\operatorname{Spec} \Z$.
In \Cref{sec:proof} we finally prove \Cref{thm:terao}.
In \Cref{sec:trees} we collect some terminology about rooted trees.
In \Cref{sec:iterators} we define recursive and tree-iterators
and in \Cref{sec:parallel} we introduce algorithms to evaluate them in parallel.
\Cref{sec:WhyHPCGAP} summarizes the merits of the high performance computing (HPC) version of \GAP, which we used to implement the above mentioned algorithms.
We conclude by giving some timings in \Cref{sec:timings} to demonstrate the significance of our parallelized algorithms in the generation of (certain classes) of simple rank $3$ matroids.

\section*{Acknowledgments}

We would like to thank Rudi Pendavingh for pointing us to the paper \cite{DW89} and providing us with his \textsf{SageMath} code to compute the Dress-Wenzel condition for representability of matroids established in loc.~cit.
Using his code we could avoid computing the empty moduli spaces of ca.\ 400.000 nonrepresentable matroids with $14$ atoms.
We would also like to thank Markus Lange-Hegermann and Sebastian Posur for their comments on an earlier version of this paper.
Our thanks goes also to Sebastian Gutsche and Fabian Zickgraf who helped us to setup the VM on which the public database is running.
Last but not least, we are indebted to the anonymous referees for their careful reading and for the valuable suggestions which helped us improve the exposition.

\section{Simple Matroids} \label{sec:matroids}

\subsection{Basic Definitions}

Finite simple matroids have many equivalent descriptions.
For our purposes we prefer the one describing the lattice of flats.
\begin{defn}
  A \textbf{matroid} $M=(E,\F)$ consists of a finite ground set $E$ and a collection $\F$ of subsets of $E$, called \textbf{flats} (of $M$), satisfying the following properties:
  \begin{enumerate}
    \item The ground set $E$ is a flat;
    \item The intersection $F_1\cap F_2$ is a flat, if $F_1$ and $F_2$ are flats;
    \item If $F$ is a flat, then any element in $E \setminus F$ is contained in exactly one flat covering $F$.
  \suspend{enumerate}
  Here, a flat is said to \textbf{cover} another flat $F$ if it is minimal among the flats properly containing $F$.
  A matroid is called \textbf{simple} if
  \resume{enumerate}
    \item it is loopless, i.e., $\emptyset$ is a flat;
    \item it contains no parallel elements, i.e., the singletons are flats, which are called \textbf{atoms}.
  \end{enumerate}
\end{defn}

For a matroid $M=(E,\F)$ and $S\subseteq E$ we denote by $r(S)$ the \textbf{rank of $S$} which is the maximal length of chains of flats in $\F$ all contained in $S$.
The \textbf{rank of the matroid $M$} is defined to be $r(E)$.
A subset $S\subseteq E$ is called \textbf{independent} if $|S|=r(S)$ and otherwise \textbf{dependent}.
A maximal independent set is called a \textbf{basis of $M$}.

\begin{rmrk}[Basis Extension Theorem]\label{rm:basis}
Any independent subset of a matroid can be extended to a basis.
Hence, the cardinality of any basis equals the rank of the matroid.
\end{rmrk}

The flats form a poset\footnote{The poset of flats is a geometric lattice, i.e., a finite atomic semimodular lattice.
Conversely, finite atomic semimodular lattices give rise to matroids.} $\F$ by inclusion.
Dually, a \textbf{coatom} is a maximal element in $\F \setminus \{ E \}$.
An isomorphism between the matroids $(E,\F)$ and $(E',\F')$ is a bijective map $E \to E'$ which induces an isomorphism $\F \to \F'$ of posets.

Originally, matroids were introduced as an abstraction of the notion of linear (in)dependence in linear algebra.

\begin{defn}
  A \textbf{central arrangement} over a field $\bF$ is a finite set $\A$ of $(n-1)$-dimensional subspaces of an $n$-dimensional $\bF$-vector space $V$.
  The \textbf{lattice of flats} $L(\A)$ is the set of intersections of subsets of $\A$, partially ordered by reverse inclusion, where the empty (set-theoretic) intersection is defined as $V$.
  The arrangement $\A$ is called \textbf{essential} if $\{0\} \in L(\A)$.
\end{defn} 
  The pair $(\A, L(\A))$ is a matroid of rank $n - \dim \bigcap_{H \in \A} H$, i.e., of rank $n$ iff $\A$ is essential.
  We call such a pair a \textbf{vector matroid} over $\bF$.
This motivates the following definition:

\begin{defn}
  A matroid is called \textbf{representable over the field $\bF$} if it is isomorphic to a vector matroid over $\bF$.
  A matroid is called \textbf{representable} if it is representable over some field $\bF$.
\end{defn}

The following matroid invariant and its specialization play an important role in our study of simple rank $3$ matroids.
\begin{defn}
The \textbf{Tutte Polynomial} $T_M(x,y)$ of a matroid $M=(E,\F)$ is defined by
\[
T_M(x,y) \coloneqq \sum_{S \in \mathcal{P}(E)} (x-1)^{r(M)-r(S)} (y-1)^{|S|-r(S)}\mbox{.}
\]
A matroid is called \textbf{Tutte-unique}, if it is determined by its Tutte polynomial, i.e. any matroid with the same Tutte polynomial is isomorphic to the given one.
An important evaluation of the Tutte polynomial is the \textbf{characteristic polynomial} \[
\chi_M(t) \coloneqq (-1)^{r(M)}T_M(1-t,0) = \sum_{S \in \mathcal{P}(E)} (-1)^{|S|}t^{r(M)-r(S)}.
\]
\end{defn}


The main application of this article is the enumeration of matroids with an integrally splitting characteristic polynomial.
We will denote the class of rank $r$ matroids with \textbf{integrally splitting characteristic polynomial} by $\mathcal{ISM}_r$.
Such a factorization of $\chi_M(t)$ is often implied by stronger combinatorial or (in the representable case) algebraic/geometric properties.
The only known converse statement is that for a graphic or cographic matroid $M$ induced by a planar graph $G$ the matroid has an integrally splitting characteristic polynomial if and only if $G$ is chordal as shown in~\cite{DK98} for graphic and in~\cite{KR11} for cographic matroids, respectively.
In both cases, the fact that the characteristic polynomial is integrally splitting even implies that $M$ is supersolvable.
However, it is still safe to say that the rather small class of matroids in $\mathcal{ISM}_r$ is not yet well understood when $r \geq 3$.

\subsection{Simple Rank $3$ Matroids}
We will restrict ourselves to the case of simple rank~$3$ matroids in the following definitions.
It is worth pointing out at this point that a rank $3$ matroid is simple if and only if it is \textbf{paving} which in general means that any circuit is at least as large as the rank of the matroid.
The smallest class we will consider is that of supersolvable matroids introduced by Stanley in~\cite{Sta72}.
In the rank $3$ case the definition can be given as follows:
\begin{defn}
A matroid $M=(E,\F)$ of rank $3$ is \textbf{supersolvable} if there exists a flat $F_0\in\F$ of rank $2$ such that every intersection with other flats of rank $2$ is nonempty.
In this case the characteristic polynomial is integrally splitting with roots
\[
  \chi_M(t) = \left(t -1\right)\left(t - (|F_0|-1)\right)\left(t-(|E|-|F_0|)\right) \mbox{.}
\]
Define $\mathcal{SSM}_3$ to be the class of all supersolvable rank $3$ matroids.
\end{defn}

To introduce the next combinatorial classes of matroids we need the notions of deletion and reduced contraction of a matroid with respect to an element $H$ of the ground set $E$.
The deletion just removes the element $H$ from the matroid and for a representable matroid the reduced contractions is the matroid that arises by intersecting all hyperplanes with $H$:

\begin{defn}
Let $M=(E,\F)$ be a matroid and $H\in E$.
Define the \textbf{deletion of $H$} to be the matroid $M' \coloneqq M\setminus H\coloneqq(E',\F')$ where
\begin{align*}
E'  \coloneqq &  E\setminus H \coloneqq E\setminus \{H\},\\
\F' \coloneqq & \F\setminus H \coloneqq\{F\setminus\{H\}\mid F\in\F\}.
\end{align*}
The \textbf{reduced contraction\footnote{This definition mimics the usual notion of restriction for hyperplane arrangements. Note that it differs from the matroid-theoretic contraction since it does not contain loops and parallel elements.} of $H$}
is the matroid $M''\coloneqq M^H \coloneqq (E'',\F'')$ where
\[
\F'' \coloneqq \F^H \coloneqq \{F \in \F \mid \{H\}\subseteq F\},
\]
and its atoms $E''=E^H$ are identified with the flats of rank $1$ in $\F^H$.
If $\{H\}$ is a flat in $M$ then $M^H$ is a simple matroid.
In particular, if $M$ is simple then so are $M\setminus H$ and $M^H$.
\end{defn}

The following two classes stem from the theory of free hyperplane arrangements.
The first one generalizes Terao's notion of inductively free hyperplane arrangements to matroids~\cite{Ter80}.
Intuitively, a matroid $M$ is inductively free if there is a pair of a deletion and reduced contraction with respect to one atom both of which are inductively free and the characteristic polynomial of the reduced contraction divides the one of $M$.
As the start of the induction one defines all rank $2$ matroids and all Boolean matroids to be inductively free.

\begin{defn}\label{def:if}
We define the class $\mathcal{IFM}_3 $ of \textbf{inductively free rank $3$ matroids} to be the smallest class of simple rank $3$ matroids containing
\begin{itemize}
\item the Boolean matroid $M_3 \coloneqq (\{1,2,3\},\mathcal{P}(\{1,2,3\}))$ and
\item $M=(E,\F)$ with $|E|>3$ if there exists an $H\in E$ such that $\chi_{M^H}(t)|\chi_M(t)$ and $M\setminus H\in \mathcal{IFM}_3$.
\end{itemize}
\end{defn}

Recently, Abe introduced a larger class of combinatorially free arrangements in~\cite{Abe16}.
This class is defined in a similar fashion just without the assumption on the deletion of a matroid.

\begin{defn}\label{def:df}
The class $\mathcal{DFM}_3 $ of \textbf{divisionally free rank $3$  matroids} is the smallest class of simple  rank $3$ matroids containing
\begin{itemize}
\item the Boolean matroid $M_3 \coloneqq (\{1,2,3\},\mathcal{P}(\{1,2,3\}))$ and
\item $M=(E,\F)$ with $|E|>3$ if there exists an $H\in E$ such that $\chi_{M^H}(t)|\chi_M(t)$.
\end{itemize}
\end{defn}

\begin{rmrk}\label{rm:strict_inclusions}
The following strict inclusions hold
\[
\mathcal{SSM}_3 \subsetneq \mathcal{IFM}_3 \subsetneq \mathcal{DFM}_3 \subsetneq \mathcal{ISM}_3 \mbox{,}
\]
where the first strict inclusion is shown in~\cite{JT84} and the second inclusion in \cite{Abe16} (for strictness of the inclusion in rank $3$ cf.~\Cref{exmp:div_free_not_ind_free} and for vector matroids of rank at least $4$ cf.~loc.~cit.).
The last inclusion holds by the definition of divisional freeness and since $\chi_M(t)=(t-1)(t-(|E|-1))$ for any simple matroid $M=(E,\F)$ of rank $2$ (for strictness see~\Cref{tbl:simple_matroids}, for example).
\end{rmrk}

\begin{rmrk} \label{rmrk:ind_div_using_db}
  Due to the recursive nature of the definition of inductive freeness, a database containing the simple rank $3$ matroids with up to $n$ atoms is extremely useful when deciding the inductive freeness of those with $n+1$ atoms.
  This is how we determined the subclass $\mathcal{IFM}_3$ in our database.
\end{rmrk}

%

\section{Generating Rank $3$ Matroids with Integrally Splitting Characteristic Polynomials} \label{sec:matroid_iter}

Since we will focus on the rank $3$ case we prefer to describe them as special instances of bipartite graphs.
And as already mentioned in the introduction, the description of tree-iterators $T_\bullet$ generating simple rank $3$ matroids will rely on the language of bipartite graphs.
Our description is a special case of $m$-partitions which describes general paving matroids of rank $m+1$~\cite[Proposition 2.1.24]{Oxley2011}.

\begin{defn}\label{defn:matroidal_bipartitegraph}
  A (proper) \textbf{$2$-partition} of a finite set $E$ is a set $\mathcal{E}$ of nonempty (proper) subsets of $E$, called blocks, such that
  \begin{enumerate}
    \item each block contains at least $2$ elements;
    \item each pair of elements is contained in exactly one block.\label{defn:matroidal_bipartitegraph.b}
  \end{enumerate}
\end{defn}
Condition \eqref{defn:matroidal_bipartitegraph.b} means that $\{ \binom{F}{2} \mid F \in \mathcal{E} \}$ is a partition of $\binom{E}{2} \coloneqq \{ \{a,b\} \subseteq E \mid a \neq b\}$.

\begin{rmrk} \label{rmrk:monoidal_partition}
Let $\mathcal{E}$ be a $2$-partition of $E$.
Then
\begin{itemize}
  \item $\bigcup \mathcal{E} = E$.
  \item $|F \cap F'| \leq 1$ for all $F, F' \in \mathcal{E}$ with $F \neq F'$.
  \item $\sum_{F \in \mathcal{E}} \binom{|F|}{2} = \binom{|E|}{2}$.
  \item The union $E \cup \mathcal{E}$ defines the vertices of a bipartite graph with adjacency given by membership.
    We call bipartite graphs admitting such a description \textbf{matroidal}, if the $2$-partition is proper.
    Connecting, as in \Cref{fig:braid_arr}, the elements of $E$ with an initial element and the blocks with a terminal element we obtain a geometric lattice of flats of a simple rank $3$ matroid.
    Hence, there is a bijective correspondence between simple rank $3$ matroids with ground set $E$ and proper $2$-partitions of $E$.
    Therefore, we will henceforth call the elements of $E$ \textbf{atoms} and those of $\mathcal{E}$ \textbf{coatoms}.
\end{itemize}
Since each pair of atoms is contained in exactly one coatom, the left hand side of the last equation counts the number of pairs of atoms which are joined by the coatoms. This count must be equal to the number of all pairs of atoms which is the right hand side of the equation.
\end{rmrk}

We divide the matroid generation by only considering those with a fixed number of coatoms of each size at a time.
To this end, we define size vectors of coatoms which satisfy the condition in~\Cref{rmrk:monoidal_partition} as multiplicity vectors:
\begin{defn}\label{def:mult_vector}
  For $n \in \N$ we call $(m_k) \coloneqq (m_k)_{k=2,\ldots,n-1} \coloneqq (m_2, \ldots, m_{n-1})$ a \textbf{multiplicity vector} of size $n$ if $\sum_{k=2}^{n-1} m_k \binom{k}{2} = \binom{n}{2}$.

Each $2$-partition $\mathcal{E}$ gives rise to an \textbf{associated multiplicity vector} $m(\mathcal{E}) = (m_k)$ with $m_k = |\{  F \in \mathcal{E} : |F| = k \}|$.
\end{defn}

An example of such a matroidal bipartite graph corresponding to the rank $3$ braid arrangement $\A_3$ is given in \Cref{fig:braid_arr}.
It has the multiplicity vector $(m_2,m_3)=(3,4)$ and the characteristic polynomial $\chi_{\A{_3}}(t)=(t-1)(t-2)(t-3)$.
\begin{center}
\begin{minipage}{0.9\linewidth}
\begin{center}
  \begin{tikzpicture}[scale=0.55]
  \node[label=left:{$\bm{x}$}] at (-10, 0)   (x) {\LARGE$\bullet$};
  \node[label=left:{$\bm{y}$}] at (-6, 0)   (y) {\LARGE$\bullet$};
  \node[label=left:{$\bm{z}$}] at (-2, 0)   (z) {\LARGE$\bullet$};
  \node[label=right:{$\bm{x-y}$}] at (2, 0)   (xy) {\LARGE$\bullet$};
  \node[label=right:{$\bm{x-z}$}] at (6, 0)   (xz) {\LARGE$\bullet$};
  \node[label=right:{$\bm{y-z}$}] at (10, 0)   (yz) {\LARGE$\bullet$};
  
  \node[label=left:{\color{blue}$\bm{3}$}] at (-12, 7)   (c1) {\LARGE$\bullet$};
  \node[label=left:{\color{blue}$\bm{3}$}] at (-8, 7)   (c2) {\LARGE$\bullet$};
  \node[label=left:{\color{blue}$\bm{3}$}] at (-4, 7)   (c3) {\LARGE$\bullet$};
   \node[label=left:{\color{blue}$\bm{3}$}] at (0, 7)   (c4) {\LARGE$\bullet$};
  \node[label=right:{\color{blue}$\bm{2}$}] at (4, 7)   (c5) {\LARGE$\bullet$};
  \node[label=right:{\color{blue}$\bm{2}$}] at (8, 7)   (c6) {\LARGE$\bullet$};
  \node[label=right:{\color{blue}$\bm{2}$}] at (12, 7)   (c7) {\LARGE$\bullet$};
  
   \node at (0, 10)   (0) {$\bullet$};
  \node at (0, -3)   (V) {$\bullet$};
  
 \draw[thick] (x.center) -- (c1.center) -- (y.center);
 \draw[thick] (xy.center) -- (c1.center);
 
 \draw[thick] (x.center) -- (c2.center) -- (z.center);
 \draw[thick] (xz.center) -- (c2.center);
 
 \draw[thick] (y.center) -- (c3.center) -- (z.center);
 \draw[thick] (yz.center) -- (c3.center);
 
 \draw[thick] (xy.center) -- (c4.center) -- (xz.center);
 \draw[thick] (yz.center) -- (c4.center);
 
 \draw[thick] (x.center) -- (c7.center) -- (yz.center);
 \draw[thick] (y.center) -- (c6.center) -- (xz.center);
 \draw[thick] (z.center) -- (c5.center) -- (xy.center);
 
 \draw[thick, dashed] (V.center) -- (x.center);
 \draw[thick, dashed] (V.center) -- (y.center);
 \draw[thick, dashed] (V.center) -- (z.center);
 \draw[thick, dashed] (V.center) -- (xy.center);
 \draw[thick, dashed] (V.center) -- (xz.center);
 \draw[thick, dashed] (V.center) -- (yz.center);
 
 \draw[thick, dashed] (0.center) -- (c1.center);
 \draw[thick, dashed] (0.center) -- (c2.center);
 \draw[thick, dashed] (0.center) -- (c3.center);
 \draw[thick, dashed] (0.center) -- (c4.center);
 \draw[thick, dashed] (0.center) -- (c5.center);
 \draw[thick, dashed] (0.center) -- (c6.center);
 \draw[thick, dashed] (0.center) -- (c7.center);
  \end{tikzpicture}
\end{center}

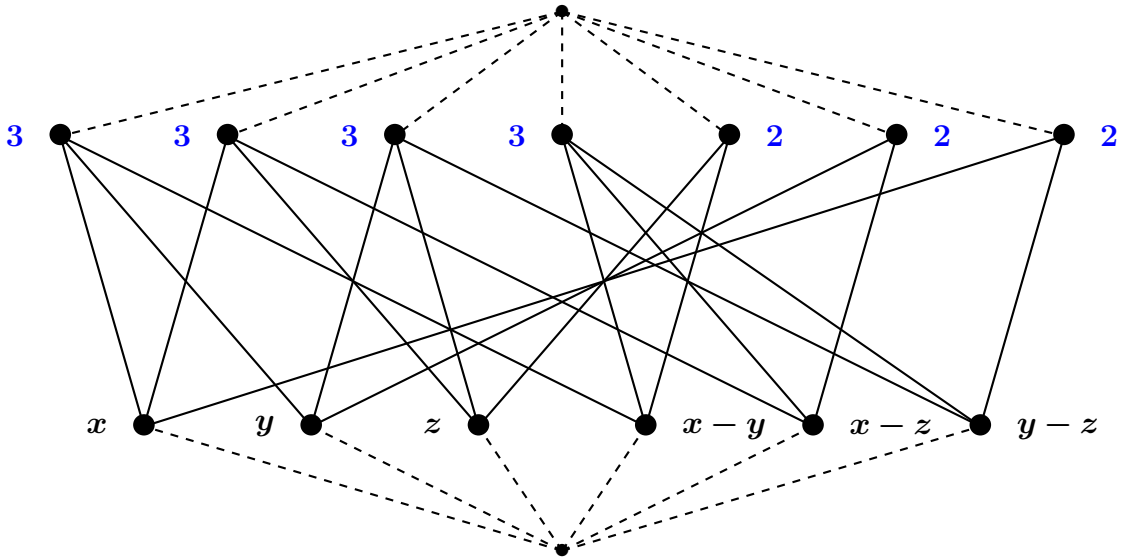
\captionof{figure}{The lattice of flats of the $\A_3$ braid arrangement with atoms in the bottom row and coatoms in the top row.
The linear forms depicted next to the atoms are a possible representations of the matroid.
The numbers denoted in blue are the multiplicities of the coatoms.}\label{fig:braid_arr}
\end{minipage}
\end{center}

To generate the multiplicity vectors of all simple rank $3$ matroids of fixed size $n$ we can naively iterate over all vectors in $\{0,\ldots,n\}^{n-2}$ satisfying the equation in \Cref{def:mult_vector}.
Additionally, we can assume that any matroid has at least as many coatoms as atoms by a theorem of de Bruijn and Erd\H{o}s~\cite{dBE48}.
Finally, we are only considering those multiplicity vectors $(m_k)$ such that the corresponding characteristic polynomial as in~\eqref{eq:chiM} is integrally splitting.

Let $T$ denote the set of isomorphism classes of simple rank $3$ matroids on $n$ \emph{unlabeled} atoms\footnote{Cf.~(\url{http://oeis.org/A058731}).}.
In what follows we will embed $T$ as the set of leaves in a rooted tree $T_\bullet$.
To make this precise we will use the language of rooted trees and tree-iterators which we summarized in \Cref{sec:trees} and \Cref{sec:iterators}, respectively.

We start by describing a rooted tree $\widetilde{T}_\bullet$ with an action of the symmetric group $\Sym(n)$, such that the quotient tree $T_\bullet \coloneqq \widetilde{T}_\bullet / \Sym(n)$ (in the sense of \Cref{exmp:orbit}) classifies the set $T$.
The set of leaves $\widetilde{T} \coloneqq \lim \widetilde{T}_\bullet$ is then the set of isomorphism classes of simple rank $3$ matroids on $n$ \emph{labeled} atoms.\footnote{Cf.~(\url{http://oeis.org/A058720}, for $k=3$).}

We subdivide this problem by describing a subtree $\widetilde{T}_\bullet^{(m_k)} \subseteq \widetilde{T}_\bullet$ such that $T_\bullet^{(m_k)} \coloneqq \widetilde{T}_\bullet^{(m_k)} / \Sym(n)$ classifies the set $T^{(m_k)}$ of all nonisomorphic simple rank $3$ matroids with given multiplicity vector $(m_k)$.
Our goal is to build a locally uniform tree-iterator $t^{(m_k)}$ having the tree $T_\bullet^{(m_k)}$ as its tree of relevant leaves in the language of \Cref{rmrk:B3}.

In order to describe the larger tree $T_\bullet^{t^{(m_k)}}$ associated to $t^{(m_k)}$ (in the sense of \Cref{rmrk:B2}) we define the theoretically possible pre-stages of $2$-partitions with associated multiplicity vector $(m_k)$ in the following sense:

\begin{defn} \label{defn:admissible}
  We call a set $A = \{ A_i \}$ of subsets of $\{1, \ldots, n\}$ an \textbf{admissible partial $2$-partition} of level $k_0 \in \{1, \ldots, n-2\}$ for a multiplicity vector $(m_k)=(m_2, \ldots, m_{n-1})$ of size $n$ if
  \begin{itemize}
    \item $|\{i : |A_i| = k\}| = m_k$ for all $k$ with $k_0 < k \leq n-1$;
    \item $|A_i \cap A_j| \leq 1$ for all $i < j$.
  \end{itemize}
\end{defn}
The rooted tree $T_\bullet^{t^{(m_k)}}$ can now be described as follows:
We set $T_0^{t^{(m_k)}} \coloneqq \{*\}$ and for $1 \leq i \leq n-2$ let $T_i^{t^{(m_k)}}$ consist of all admissible partial $2$-partitions of level $k_0 = n - i - 1$ modulo the action of $\Sym(n)$.
All maps $T_i^{t^{(m_k)}} \leftarrow T_{i+1}^{t^{(m_k)}}$ are evident and surjective.
The rooted tree $T^{t^{(m_k)}}$ differs from its subtree $T^{(m_k)}$ by the possible dead ends, i.e., by those admissible partial $2$-partitions that cannot be completed to a proper $2$-partition.

To describe the iterator $t^{(m_k)}$ we propose an algorithm that takes as input an admissible partial $2$-partition $A$ of some level~$k_0$ and iterates over all possible extensions to admissible partial $2$-partition of the next nontrivial smaller level $k_1$ with $k_1 < k_0$.
In this computation we only consider lexicographically minimal extensions with respect to the stabilizer of $A$ under the action of the symmetric group $\Sym(n)$ to avoid iterating over isomorphic matroids multiple times.
The details of this procedure are given in Algorithm~\ref{algo:state} (\textbf{IteratorFromState}).

Finally, to build the tree-iterator of all simple rank $3$ matroids with $n$ atoms and multiplicity vector $(m_k)$ (as bipartite graphs) we apply \textbf{IteratorFromState} to the initial state
\[
  s^{(m_k)} \coloneqq \big(n, (m_k), k_0 \coloneqq \max\{ k \mid m_k > 0 \}, A \coloneqq () \big) \mbox{.}
\]

For the proof of \Cref{thm:terao} we are only interested in those matroids with an \textbf{integrally splitting characteristic polynomial} (see \Cref{sec:proof}).
Since the multiplicity vector of a rank $3$ matroid $M$ determines the characteristic polynomial $\chi_M$ by \eqref{eq:chiM} we only consider tree-iterators $t^{(m_k)}$ such that the corresponding characteristic polynomial defined by \eqref{eq:chiM} is integrally splitting.

\bigskip
\begin{algorithm}[H]
\SetKwIF{If}{ElseIf}{Else}{if}{then}{elif}{else}{}%
\DontPrintSemicolon
\SetKwProg{IteratorFromState}{IteratorFromState}{}{}
\SetKwProg{Next}{Next}{}{}
\LinesNotNumbered
\KwIn{
state $s$ consisting of
  \begin{itemize}
    \item a number $n$ of atoms
    \item a multiplicity vector $(m_k)=(m_2, \ldots, m_{n-1})$ of size $n$ \tcp*{cf.~\Cref{sec:Intro}}
    \item an integer $n-1 \geq k_0 \geq 2$ with $m_{k_0} > 0$
    \item an admissible partial $2$-partition $A$ of level $k_0$ for $(m_k)$ \tcp*{cf.~\Cref{defn:admissible}}
  \end{itemize}
}
\KwOut{tree-iterator $\mathtt{iter}$ for which \textbf{Next}($\mathtt{iter}$) returns one of the following:
\begin{itemize}
  \item \textbf{IteratorFromState}(state satisfying above specifications for $k_1$ defined in line~\ref{line:k_1}),
  \item adjacency list, or
  \item $\mathtt{fail}$
\end{itemize}
}
\IteratorFromState(){($s := ( n, (m_k), k_0, A)$, $S$)}{
    \nl  Initialize an iterator $\mathtt{iter}$ and equip it with\;
      \nl  \, \textbullet\, an empty list $APP$ to store the produced admissible partial $2$-partitions,\;
      \nl  \, \textbullet\, an integer $k_1 \coloneqq \max (\{1\} \cup \{k' < k_0 \mid m_{k'} > 0\}) \geq 1$, and\; \label{line:k_1}
      \nl  \, \textbullet\, a function \textbf{Next} as defined in line~\ref{line:Next}\;
      \nl  \label{line:Next} \Next(){($\mathtt{iter}$)}{
       \tcc{
       find the next block of coatoms of multiplicity $k_0$:
       }
        \nl  \uIf{next $A' = \{A'_1, \ldots, A'_{m_{k_0}} \}$ exists with  \tcp*{find $m_{k_0}$ new coatoms}  \label{line:A'}
        \nl  \, \textbullet\, $A \cup A'$ admissible partial $2$-partition of level $k_1$\; \label{line:admissible}
        \tcc{
        the following line guarantees the generation of pairwise nonisomorphic bipartite graphs, the justification will be provided in \Cref{rmrk:sym_red}
        }
        \nl  \, \textbullet\, the lexicographically minimal element $A''$ in the orbit of $A'$ under $\Stab_{\Sym(n)}(A)$ is not contained in $APP$\label{algo:ln:min_image} \; \label{line:sym}
        \tcc{
        Lines \ref{line:A'},\ref{line:admissible},\ref{line:sym} can again be realized by an iterator which returns the next $A'$ or $\mathtt{fail}$ if no such $A'$ exists.
        }
        }{
        \nl  save $A''$ in $APP$\;
        \nl  $A'' := A \cup A''$ \tcp*{augment the current partial $2$-partition}
        \nl  \uIf{$k_1 \geq 2$}{
        \nl  $s' \coloneqq (n, (m_k), k_1, A'')$ \tcp*{define the new state}
        \tcc{return \textbf{IteratorFromState} applied to the new state $s'$}
        \nl  \Return{{\upshape{\textbf{IteratorFromState}}($s'$)}}}
        \nl  \Else{
        \nl \Return{$A''$} \tcp*{return the complete adjacency list}
        }
        }
       \nl  \Else{
         \nl  \Return{$\mathtt{fail}$}
       }
       }
  \Return{$\mathtt{iter}$}
 }
\caption{IteratorFromState \label{algo:state}}
\end{algorithm}

\bigskip

We have implemented Algorithm~\ref{algo:state} as part of the \GAP-package $\mathtt{MatroidGeneration}$ \cite{MatroidGeneration}.
We show in \Cref{sec:parallel} how to evaluate recursive iterators in parallel.
We applied Algorithm~\ref{algo:state} to all multiplicity vectors with an integrally splitting characteristic polynomial and stored the resulting matroids in the database \cite{matroids_split} using the \GAP-package $\mathtt{ArangoDB}$\-$\mathtt{Interface}$ \cite{ArangoDBInterface}.

\bigskip
\begin{rmrk} \label{rmrk:sym_red}
  Line~\ref{algo:ln:min_image} in Algorithm~\ref{algo:state} ensures that the iterator $\mathtt{iter}$ instantiated by the state $s$ does not create two isomorphic adjacency lists $A''$ and $A_2''$ with a common sublist $A$.
  Furthermore, the lexicographically minimal element of the orbit of\footnote{$A \cup A'$ is considered in line~\ref{line:admissible}} $A \cup A'$ under $\Stab_{\Sym(n)}(A)$ is nothing but $A \cup A''$, namely the union of $A$ and the lexicographically minimal element $A''$ of the orbit of $A'$ under $\Stab_{\Sym(n)}(A)$ (considered in line~\ref{line:sym}).
  This is due to the fact that sets in $A'$ are of different cardinality than those in $A$.
\end{rmrk}

\begin{rmrk} \label{rmrk:balanced}
A simple rank $3$ matroid $M$ (of size $n$) is \emph{weakly atom balanced on dependent coatoms}, i.e., every atom is contained in at most $\frac{n-1}{2}$ coatoms of cardinality at least $3$ (these are all dependent in $M$).
\end{rmrk}
\begin{proof}
Let $M$ be a simple rank $3$ matroid of size $n$ and consider a fixed atom $k$.
Let $F_1,\dots,F_\ell$ be the coatoms of size at least $3$ containing the atom $k$.
This implies $|F_i\setminus\{k\}|\ge 2$ for all $1\le i \le \ell$.
Furthermore, by definition of a simple matroid it holds that $F_i\cap F_j=\{k\}$ for all pairs $1\le i < j \le \ell$.
This means that the coatoms $F_1,\dots,F_\ell$ contain each at least two atoms of the $n-1$ atoms which are different from $k$ and moreover these additional atoms are all pairwise distinct.
This immediately yields $\ell \le \frac{n-1}{2}$ which proves the claim.
\end{proof}

\begin{rmrk}
  The computationally difficult part of Algorithm~\ref{algo:state} is to find admissible completions of partial $2$-partitions in Line~\ref{line:A'}.
  Na\"ively, one needs to loop over all subsets of $\{1,\dots,n\}$ of a given size and discard all those which contain a pair of atoms which is already contained in a previous coatom.
  To speed up this part of the computation we use the following methods:
  \begin{itemize}
  	\item Following \Cref{rmrk:balanced} we can discard atoms in computations of multiplicity $k_0\ge 3$ if they are already contained in $\frac{n-1}{2}$ coatoms.
  	\item We can assume that within one level the coatoms are ordered lexicographically.
  	Thus, we discard all subsets of $\{1,\dots,n\}$ that are lexicographically smaller than the last coatom of the same size before calling the search algorithm.
	\item If at any step a transposition $(e_1 \: e_2)$ is in the stabilizer of all previous coatoms we treat the atoms $e_1$ and $e_2$ as equivalent at this step of the algorithm.
	This means that we loop over subsets of $\{1,\dots,n\}$ modulo all such equivalences instead of the entire set $\{1,\dots,n\}$.
	To ensure we do not miss relevant subsets we also need to consider subsets of these representatives of smaller sizes and fill up the resulting subsets to the correct size with atoms that are not contained in any coatom yet.
	We can choose the first unused atoms in the lexicographic order for this matter.
  \end{itemize}
\end{rmrk}

\begin{rmrk} \label{rmrk:min_img}
To calculate lexicographically minimal elements of orbits we use the \texttt{Ferret} and \texttt{Images} packages, by the third author:
  \begin{itemize}
    \item \texttt{Ferret} is a reimplementation of Jeffrey Leon's Partition Backtrack Algorithm \cite{Leo91}, with a number of extensions \cite{JPW18}.
    \item \texttt{Images} provides algorithms which, given a permutation group \(G\) on a set \(\Omega\) and a set \(S \subseteq \Omega\), find the lexicographically minimal image of \(S\) under \(G\), or a canonical image of the orbit of \(S\) under \(G\). \texttt{Images} uses the algorithms of Jefferson et al.~\cite{JJPW18}.
  \end{itemize}
  For this project, both \texttt{Ferret} and \texttt{Images} were extended to be compatible with \textsf{HPC-GAP}.
\end{rmrk}

\begin{rmrk}\label{rmrk:find_matroids}
In the database~\cite{matroids_split} we store lexicographically minimal elements of the list of coatoms computed by the \textsf{GAP} package \texttt{Images}.
This specific form enables the lookup of an arbitrary matroid in the database by computing its uniquely defined minimal image under the action of the symmetric group.
We have also used this lookup procedure to compute the inductive freeness property since this property depends by definition on the inductive freeness of matroids of smaller size (cf. \cref{rmrk:ind_div_using_db}).
\end{rmrk}

\begin{rmrk}\label{rm:not_terminated}
The computations to generate all simple rank $3$ matroids with integrally splitting characteristic polynomial terminated on all $695$ possible multiplicities vectors except for the two vectors $(m_3,m_4,m_5)=(21,3,1)$ and $(m_2,m_3,m_4,m_5,m_6)=(1,23,1,0,1)$.
The latter multiplicity vector is in any case not relevant for Terao's conjecture as any matroid with this multiplicity vector would not be coatom balanced (cf.~\cref{def:balanced}).
In \Cref{prop:no_matroid}, we prove that there are no matroids with one of the above multiplicities vectors.
Hence, these computations which did not terminate do not impose any restrictions on \Cref{thm:terao} or \Cref{tbl:simple_matroids}.
\end{rmrk}

\begin{prop}\label{prop:no_matroid}
Let $v_1$ and $v_2$ be the multiplicity vectors $(m_3,m_4,m_5)=(21,3,1)$ and $(m_2,m_3,m_4,m_5,m_6)=(1,23,1,0,1)$ respectively.
Then, there exists no simple rank $3$ matroid of size $14$ having either $v_1$ or $v_2$ as its associated multiplicity vector.
\end{prop}
\begin{proof}
Given an admissible partial $2$-partition $A$ and an atom $e$ we denote by $d_A(e)$ the \emph{deficiency} of $e$ in $A$ which is the number of atoms that are not contained in a common coatom with $e$ in $A$.
For both multiplicity vectors we investigate the admissible partial $2$-partitions that contain all coatoms of size greater than $3$.
We will argue based on the parity of their deficiencies that all of them cannot be completed to a matroid with the remaining coatoms of size $3$ (and one coatom of size $2$ in the case of $v_2$) which completes the proof.

Consider a step in which we add the coatom $C\coloneqq \{e_1,e_2,e_3\}$ to a list of coatoms $A$ and obtain a new list $A'$.
Then we have $d_{A'}(e_i)=d_A(e_i)-2$ for $1\le i\le 3$ and the remaining deficiencies remain constant.
In particular, the parity of all deficiencies is constant in this step.

In the case of the multiplicity vector $v_1$ we can without loss of generality assume that all admissible partial $2$-partitions with all coatoms  of size greater than $3$ contain the coatom $\left[1,\dots,5\right]$ and the atom $4$ is not contained in any coatom of size $4$.
Thus, we have $d_A(4)=13-4=9$ for all such lists $A$.
Since this number is odd but the deficiency of any matroid is~$0$ the above discussion proves that all such partial list of coatoms of $v_1$ can not be completed to a matroid.

To prove the remaining statement regarding the  multiplicity vector $v_2$, we start our parallel matroid generation algorithm but terminate after completing all levels of size greater than $3$.
We obtain the two partial admissible $2$-partitions
\[
A_1 \coloneqq	[ [ 1, 2, 3, 4, 5, 6 ], [ 1, 7, 8, 9 ]], \quad A_2 \coloneqq [ [ 1, 2, 3, 4, 5, 6 ], [ 7, 8, 9, 10 ] ].
\]
Now, we need to add coatoms of size $3$ and exactly one coatom of size $2$ to the admissible partial $2$-partitions $A_{1},A_{2}$.
An analogous argument as in the first case shows that the number of atoms with odd deficiency of the lists $A_{1},A_{2}$ must be exactly two.
Computing deficiencies of the atoms in the lists $A_1$ and $A_2$ yields that $1$ is the only atom with an odd deficiency in $A_1$ whereas all atoms in $A_2$ have an even deficiency.
Thus, there exists no matroid with multiplicity vector~$v_2$.
\end{proof}

\section{How to Decide Representability of a Matroid?} \label{sec:representability}

The Basis Extension Theorem for matroids (cf.~\Cref{rm:basis}) implies that the (possibly empty) space $\mathcal{R}(M)$ of \emph{all} representations (over some unspecified field $\bF$) of a matroid $M = (E, \F)$ is an \emph{affine variety}, namely an affine subvariety $V(I') \subseteq \AA_\Z^{rn + 1}$, where $r$ is the rank of $M$ and $n$ its number of atoms.

More precisely, let $\AA^{rn+1}_\Z \coloneqq \Spec R[d]$, where $R \coloneqq \Z[ a_{ij} \mid i = 1, \ldots, r,\, j = 1, \ldots, n ]$ and $d$ a further indeterminate.
To describe the ideal $I'$ set $A \coloneqq (a_{ij}) \in R^{r \times n}$.
For a subset $S \subseteq E$ denote by $A_S$ the submatrix of $A$ with columns in $S$.
Further, let $\mathcal{B}(M) = \{ B_1, \ldots, B_b \}$ be the set of bases of $M$.
Then
\[
  I' = \left\langle \det(A_D) \mid D \subseteq E  \mbox{ dependent}, |D|=r  \right\rangle + \left\langle 1 -d \prod_{B \in \mathcal{B}(M)} \det(A_B) \right\rangle  \unlhd R[d] \mbox{.}
\]
It follows that $M$ is representable (over some field $\bF$) if and only if $1 \notin I'$.
This ideal membership problem can be decided by computing a Gröbner basis of $I'$.
This is basically the algorithm suggested in \cite{Oxley2011}.

If the ideal $I'$ is a maximal ideal in $R[d]$ the moduli space of representations $\Spec R[d]/I'$ of the matroid $M$ contains only one point.
In this case, the matroid $M$ has a unique representation (up to equivalence) and we call $M$ \emph{uniquely representable over $\Spec \Z$}.

However, it is computationally more efficient to represent $\mathcal{R}(M)$ as a
 quasi-affine set $V(I) \setminus V(J) \subseteq \AA_\Z^{rn} = \Spec R$, where $J$ is a principal ideal.
Denote by $J_S \coloneqq \langle \det( A_S ) \rangle$ the principal ideal generated by the maximal minor corresponding to $S$, provided $|S| = r$.
Then
\begin{align*}
  I &= \sum \{ J_D \mid D \subseteq E  \mbox{ dependent}, |D|=r \}, \\
  J &= \prod \{ J_B \mid B \in \mathcal{B}(M) \}.
\end{align*}
In particular, $J$ is a principal ideal.
It follows that $M$ is representable (over some field $\bF$) iff $\det(A_S) \notin \sqrt{I}$ for all $S \subseteq E$ basis.
The ideal $I$ can be replaced by the saturation
\[
\widetilde{I} \coloneqq I : \left(\prod_{B \subseteq E \text{ basis}}\det(A_B)\right)^\infty = I : \det(A_{B_1})^\infty : \cdots : \det(A_{B_b})^\infty \mbox{.}
\]
Then $M$ is representable iff $1 \notin \widetilde{I}$.
For the Gröbner basis computations over $\Z$ we used \textsc{Singular} \cite{Singular412} from within the GAP package $\mathtt{ZariskiFrames}$ \cite{ZariskiFrames}, which is part of the $\mathtt{CAP/homalg}$ project \cite{homalg-project,BL,GPSSyntax}.

We used a more efficient approach which does not involve working over $\AA^{rn+1}_\Z$ but fixes certain values of the matrix $A$ to $0$ or $1$ as described in~\cite[p. 184]{Oxley2011}.
Firstly, we choose a basis $B\in \mathcal{B}(M)$ and fix the corresponding submatrix $A_B$ to be the unit matrix.
Without loss of generality we can assume $B=\{1,\dots,r\}$.
Secondly, we consider the fundamental circuits with respect to this basis $B$, i.\ e.\ for each $k\in E\setminus B$ let $C(k,B)$ be the unique circuit of the matroid $M$ contained in $B\cup k$.
The entries of $A$ in the column $k\in E\setminus B$ which do not appear in $C(k,B)$ can be fixed to $0$.
Lastly, the first nonzero entry in every column and the first nonzero entry in every row of $A$ can be taken as $1$ by column and row scaling respectively.
We have added this algorithm to $\mathtt{alcove}$ \cite{alcove}.

For another approach to the rational moduli space cf.~\cite{Cun11}.

\section{Proof of \Cref{thm:terao}} \label{sec:proof}

If a matroid has an atom which is contained in many coatoms or conversely a coatom which contains many atoms any realization satisfies Terao's conjecture.
This statement will be a crucial ingredient in the proof of ~\Cref{thm:terao}.
To formalize it we make the following definition.

\begin{defn}\label{def:balanced}
Let $M$ be a simple matroid of rank $3$ and assume $\chi_M(t)=(t-1)(t-a)(t-b)$ for some integers $a,b\in\Z$ such that $a\le b$.
\begin{itemize}
\item We call $M$ \textbf{atom balanced} if each atom is contained in at most $a$-many coatoms.
\item We call $M$ \textbf{coatom balanced} if each coatom contains strictly less than $a$-many atoms.
\item If $M$ is both atom and coatom balanced we call it \textbf{strongly balanced}.
\end{itemize}
\end{defn}

The importance of balancedness in our context stems from the next proposition.

\begin{prop}\label{prop:balancedness}
Let $M$ be a simple matroid of rank $3$ and assume $\chi_M(t)=(t-1)(t-a)(t-b)$ for some integers $a,b\in\Z$ such that $a\le b$.
If $M$ is \emph{not} strongly balanced then the freeness of any arrangement of hyperplanes representing $M$ can be decided combinatorially.
These representations therefore satisfy Terao's freeness conjecture.
\end{prop}
\begin{proof}
To begin assume that $M$ is not atom balanced for some atom $A$ which is contained in $n_{M,A}$ many coatoms with $n_{M,A} > a$.
Then, Theorem 1.1 and Corollary 1.2 in~\cite{Abe14} show that any representation of $M$ is free if and only if $n_{M,A}\in \{a+1,b+1\}$.

Instead assume that $M$ is not coatom balanced.
In this case, Lemma 2.10 in~\cite{ACKN16} shows that $M$ cannot be atom balanced either which finishes the proof by the first part.
\end{proof}

Now we have all ingredients to prove \cref{thm:terao}.

\begin{proof}[Proof of \cref{thm:terao}]
It suffices to check Terao's freeness conjecture for all representations of matroids of size $14$ which do not fall into any of the following classes of arrangements for which Terao's conjecture is known to be true:
\begin{itemize}
\item If the characteristic polynomial of the arrangement is not integrally splitting the arrangement is combinatorially nonfree by Terao's Factorization Theorem~\cite{Ter81}.
\item Representations of nonstrongly balanced simple rank $3$ matroids satisfy Terao's conjecture by \cref{prop:balancedness}.
\item Any representation of an inductively free matroid is a free arrangement~\cite{Ter80}.
\item If a matroid has a unique representation over the integers\footnote{i.e., the moduli space $\Spec R / \widetilde{I} \to \Spec \Z$ of representations is $\Spec \mathbb{F}_p  \to \Spec \Z$, a singleton.} it trivially satisfies Terao's conjecture.
\end{itemize}
Querying the database \cite{matroids_split} there are
\begin{itemize}
  \item $783280$ rank $3$ matroids of size $14$ with integrally splitting characteristic polynomial,
  \item $1574$ thereof are representable over some field,
  \item $174$ thereof are not inductively free,
  \item $64$ thereof are strongly balanced.
\end{itemize}
All of these $64$ remaining matroids have $\Spec \mathbb{F}_5$ as their moduli space, i.e., each of them is only representable over fields of characteristic $5$ and any such representation is equivalent to a unique representation over $\mathbb{F}_5$; hence they are all irrelevant for Terao's freeness conjecture.
This completes the proof.
\end{proof}

\begin{rmrk}
The situation of matroids of size $14$ is surprisingly simple in that respect.
This is not the case for matroids of smaller size since there are $9$ matroids which avoid all of the above classes and exhibit a nontrivial moduli space of representations (among them the example of a free but not rigid arrangement of size $13$ described in~\cite{ACKN16}).
We will describe their moduli spaces over $\Spec \Z$ and the nonfree locus therein in a subsequent article \cite{BK} which will establish Terao's conjecture for rank $3$ arrangements with up to $14$ hyperplanes in any characteristic.
\end{rmrk}

\newpage

\appendix

\section{Rooted Trees} \label{sec:trees}

In this Appendix we discuss rooted trees and give simple examples for their use as a tool to iterate over desired sets.
Instead of the classical definition of rooted trees we use an alternative mathematical model of rooted trees in which one can easily interpret the data structure of tree-iterators and their evaluations which we introduce in \Cref{sec:iterators}.
Expressed in this model, the (parallelized) evaluation of tree-iterators (Algorithm~\ref{algo:peri}) can then be understood as a limiting process.

For the design of our algorithms we represent a finite \textbf{rooted forest} (or set of \textbf{rooted trees}) as a finite sequence of the form
\[
  T_\bullet: T_0 \xleftarrow{\phi_1} T_1 \xleftarrow{\phi_2} T_2 \xleftarrow{\phi_3} \cdots \xleftarrow{\phi_d} T_d \mbox{,}
\]
where $T_i$ is the finite set of \textbf{vertices} of \textbf{depth} $i$.
We call $d$ the \textbf{depth} of $T_\bullet$.
In particular, $T_0$ is the set of \textbf{roots}.
We denote the set of \textbf{leaves} of $T_\bullet$ by $T := \lim T_\bullet$, which is the set of non-images in $T_\bullet$.\footnote{The notation $\lim T_\bullet$ can be justified as follows: $T_\bullet$ is a sequential inverse system in the category of finite sets with the set of leaves as its limiting object.}
As mentioned in the introduction we then say that $T_\bullet$ \textbf{classifies} $T$.

A forest of rooted trees can be understood as a single rooted tree by adding a constant map $T_{-1} := \{*\} \leftarrow T_0$ and then increase all indices by $1$.
\begin{conv}
So without loss of generality we will henceforth assume $T_\bullet$ to be a rooted tree of depth $d$, i.e., $T_0 = \{*\}$ a singleton.
\end{conv}

If all maps in the inverse system are surjective then the natural map $T_d \leftarrow T$ (which is part of the limit datum) is bijective and the set leaves $T = T_d$.
In this case all leaves have the same depth $n$ and we call $T_\bullet$ \textbf{uniform (of depth $d$)}.

More generally, we call a tree $T_\bullet$ \textbf{locally uniform} if each vertex that has a leaf as a child only has leaves as children, i.e., if for each vertex $v$ of depth $i$ the following holds: $\phi_i^{-1}(v) \cap T \neq \emptyset \implies \phi_i^{-1}(v) \subseteq T$.


Many inequivalent representations of such rooted trees classifying the same set $T$ might exist:
\Cref{exmp:matched,exmp:magma} are inequivalent families of rooted trees $T^{(n)}_\bullet$ (indexed by a natural number $n$) classifying the same family of sets $T^n$ of cardinality $C_n$, the $n$-th Catalan number.

\begin{exmp}[Matched parentheses] \label{exmp:matched}
  For $i \in \N$ denote by $T_i$ the set containing $i+1$  pairs of correctly matched parentheses:
  \[
    T_0 := \{ () \}, T_1 := \{ (()), ()() \}, T_2 := \{ ()(()), (()()), ((())), (())(), ()()() \}, \ldots
  \]
  Define $T_{i-1} \xleftarrow{\phi_i} T_i$ to be the map removing the left most\footnote{or right most, ...} pair of parentheses containing no other ones.
  For a fixed $n \in \N$ the sequence $T_\bullet: T_0 \xleftarrow{\phi_1} T_1 \xleftarrow{\phi_2} T_2 \xleftarrow{\phi_3} \cdots \xleftarrow{\phi_{n-1}} T_{n-1}$ is a finite rooted tree of uniform depth $n-1$. 
  The cardinality of the set of leaves $\lim T_\bullet = T_{n-1}$ is the $n$-th Catalan number\footnote{Cf.~(\url{http://oeis.org/A000108}).} $C_n = \binom{2n}{n} -  \binom{2n}{n+1} = \frac{1}{n+1}\binom{2n}{n}$.
\end{exmp}

\begin{exmp}[Magma evaluation] \label{exmp:magma}
  For $n \in \N_{>0}$ denote by $T^{(n)}$ the set of all possible ways to evaluate the product of the sorted list of free generators of a free magma $M_n = \langle a_0,\ldots, a_n \rangle$ of rank $n+1$:
  \[
      \begin{array}{c|c|c|c}
      n & 1 & 2 & 3 \\
      \hline \hline
      M_n \rule{0em}{1.2em} & \langle a, b \rangle & \langle a, b, c \rangle & \langle a, b,c, d \rangle \\[0.3em]
      T^{(n)} &  {\color{gray} \{} ab {\color{gray} \}} & {\color{gray} \{} (ab)c, a(bc) {\color{gray} \}} & {\color{gray} \{} ((ab)c)d, (a(bc))d, (ab)(cd), a((bc)d), a(b(cd)) {\color{gray} \}}
    \end{array}
  \]
  The set $T_i^{(n)}$ for $i \in \N$ arises from $T^{(n)}$ by deleting all pairs of parentheses of depth higher than $i$.
  The maps $T^{(n)}_{i-1} \xleftarrow{\phi_i} T^{(n)}_i$ are evident.
  \[
      \begin{array}{c|c|c|c}
      n & 1 & 2 & 3 \\
      \hline \hline
      T^{(n)}_0 \rule{0em}{1.2em} & {\color{gray} \{} ab {\color{gray} \}} & {\color{gray} \{} {\color{gray} abc} {\color{gray} \}} & {\color{gray} \{} {\color{gray} abcd} {\color{gray} \}} \\[0.3em]
      T^{(n)}_1 & & {\color{gray} \{} (ab)c, a(bc) {\color{gray} \}} & {\color{gray} \{} {\color{gray} (abc)d}, (ab)(cd), {\color{gray} a(bcd)} {\color{gray} \}} \\[0.3em]
      T^{(n)}_2 & & & {\color{gray} \{} ((ab)c)d, (a(bc))d, a((bc)d), a(b(cd)) {\color{gray} \}}
    \end{array}
  \]
  The gray entries in the above table are the internal nodes of the rooted tree $T^{(n)}_\bullet$.
  The latter is not locally uniform for $n\geq 3$.
  The set of leaves $\lim T^{(n)}_\bullet$ coincides with $T^{(n)}$, by construction.
  The cardinality of $T^{(n)}$ is again the $n$-th Catalan number $C_n$.
\end{exmp}

In the following example the sets of leaves are themselves sets of rooted trees.
We hope this does not cause confusion.


\begin{exmp}[Phylogenetic trees with labeled leaves] \label{exmp:phylo}
  A phylogenetic tree is a labeled  rooted tree. 
  A phylogenetic tree with $n \in \N_{>0}$ leaves corresponds to a total partition of $n$.
  Let $T^{(n)}$ be the set of phylogenetic trees with $n$ (labeled) leaves.\footnote{Cf.~(\url{http://oeis.org/A000311}).}
  {\scriptsize
  \[
      \begin{array}{c|c|c|c}
      n & 1 & 2 & 3 \\
      \hline \hline
      T^{(n)} \rule{0em}{1.7em} & {\color{gray} \Big\{} \{1\} {\color{gray} \Big\}} & {\color{gray} \Big\{} \{\{1\},\{2\}\} {\color{gray} \Big\}} & {\color{gray} \Big\{} \{\{1\},\{2\},\{3\}\};  \{\{1\},\{\{2\},\{3\}\}\}; \{\{2\},\{\{1\},\{3\}\}\}; \{\{3\},\{\{1\},\{2\}\}\} {\color{gray} \Big\}}
    \end{array}
  \]
  }
  Truncating a phylogenetic tree at depth $i$ means to contract all edges below depth $i$ and multi-label the new leaves at depth $i$ by all their child leaves.
  For $i \in \N$ denote by $T^{(n)}_i$ the set of all truncations of trees in $T^{(n)}$ at depth $i$.
  Again, all maps $T^{(n)}_{i-1} \xleftarrow{\phi} T^{(n)}_i$ are evident.
  {\small
  \[
      \begin{array}{c|c|c|c}
      n & 1 & 2 & 3 \\
      \hline \hline
      T^{(n)}_0 \rule{0em}{1.7em} & {\color{gray} \Big\{} \{1\} {\color{gray} \Big\}} & {\color{gray} \Big\{} {\color{gray} \{1,2\}} {\color{gray} \Big\}} & {\color{gray} \Big\{} {\color{gray} \{1,2,3\}} {\color{gray} \Big\}}  \\
      T^{(n)}_1 & & {\color{gray} \Big\{} \{\{1\},\{2\}\} {\color{gray} \Big\}} & {\color{gray} \Big\{} \{\{1\},\{2\},\{3\}\}; {\color{gray} \{\{1\},\{2,3\}\}}; {\color{gray} \{\{2\},\{1,3\}\}}; {\color{gray} \{\{3\},\{1,2\}\}} {\color{gray} \Big\}}  \\
      T^{(n)}_2 & & & {\color{gray} \Big\{} \{\{1\},\{\{2\},\{3\}\}\}; \{\{2\},\{\{1\},\{3\}\}\}; \{\{3\},\{\{1\},\{2\}\}\} {\color{gray} \Big\}}
    \end{array}
  \]
  }
  The rooted tree $T^{(n)}_\bullet$ is not locally uniform for $n \geq 3$. 
  The set of leaves $\lim T^{(n)}_\bullet$ coincides with $T^{(n)}$, by construction.
\end{exmp}

Factoring out symmetries of rooted trees again yields rooted trees:
\begin{rmrk}[Rooted trees of group orbits] \label{exmp:orbit}
  Let $G$ be a group.
  A rooted tree $T_\bullet$ is called a \textbf{rooted $G$-tree} if each $T_i$ is a $G$-set and all maps $\phi_i$ are $G$-equivariant.
  A rooted $G$-tree $\lim T_\bullet$ induces a rooted tree of orbits $T_\bullet / G$.
  Furthermore $\lim (T_\bullet/G) = \lim(T_\bullet)/G$, naturally.
\end{rmrk}

\begin{exmp}[Phylogenetic trees with nonlabeled leaves]
  Applying \Cref{exmp:orbit} to the previous \Cref{exmp:phylo} yields a rooted tree classifying phylogenetic trees with unlabeled leaves.
  More precisely, the action of $\Sym(n)$ on $\{1,\ldots,n\}$ turns the rooted tree $T_\bullet$ in \Cref{exmp:phylo} into a rooted $\Sym(n)$-tree.
  The rooted tree of orbits $T_\bullet/\Sym(n)$ then classifies $T/\Sym(n)$ which is the set of phylogenetic trees with unlabeled leaves.\footnote{Cf.~(\url{http://oeis.org/A000669}).}
\end{exmp}

Our primary family of examples of rooted tree was discussed in \Cref{sec:matroid_iter}.
They have rank $3$ matroids as their set of leaves.

\section{Recursive Iterators and Tree-Iterators} \label{sec:iterators}

In this appendix we introduce the data structure of so-called tree-iterators, which we use to recursively iterate over the vertices of a rooted tree.
This data structure is a central ingredient of all algorithms.

\begin{defn} \label{defn:recursive_iterator}
  Let $T$ be a set.
  \begin{itemize}
  \item A \textbf{recursive iterator $t$ within $T$} is an iterator which upon popping produces either $\mathtt{Next}(t) = \mathtt{fail} \notin T$ or a \textbf{child} $\mathtt{Next}(t)$ which is either
    \begin{enumerate}
      \item a new recursive iterator within $T$, or
      \item an element of $T$.
    \end{enumerate}
    If the pop result $\mathtt{Next}(t)$ is $\mathtt{fail}$ then any subsequent pop result of $t$ remains $\mathtt{fail}$.
    We call $T$ the \textbf{ambient set} of $t$.
  \item A \textbf{full evaluation} of a recursive iterator recursively pops all recursive iterators until each of them pops $\mathtt{fail}$.
  \item If $t$ is a recursive iterator then the subset of elements $T(t) \subseteq T$ produced upon full evaluation is called the \textbf{set of leaves of $t$ in $T$}.
    We say that $t$ \textbf{classifies} $T(t) \subseteq T$.
  \item A recursive iterator is called \textbf{locally uniform} if every descendant either pops recursive iterators or leaves, exclusively (if not $\mathtt{fail}$).
  \item A recursive iterator $t$ within $T$ is called a \textbf{tree-iterator} if upon full evaluation each element of $T(t) \subseteq T$ is the pop result of exactly one descendant of $t$.
  \end{itemize}
\end{defn}

In order to iterate over a tree $T_\bullet$ with set of leaves $T = \lim T_\bullet$ it is somewhere between convenient and almost unavoidable to construct a tree-iterator $t$ within $T$ which might iterate over a larger tree $T^t_\bullet$ having a set of leaves $T^t = \lim T^t_\bullet$ which is larger than $T(t)$, i.e., with dead ends being all tree-iterators which are descendants of $t$ but have no own descendants.
In our application to the classification of simple rank $3$ matroids the dead ends are the admissible partial $2$-partitions which cannot be completed to a proper $2$-partition (cf.~\Cref{defn:admissible}).

\begin{rmrk} \label{rmrk:B2}
  A tree-iterator $t$ within $T$ or any of its descendants can be understood as a vertex of the rooted tree
    \[
      T^t_\bullet: T^t_0 \coloneqq \{t\} \xleftarrow{\phi^t_1} T^t_1 \xleftarrow{\phi^t_2} T^t_2 \xleftarrow{\phi^t_3} \cdots \xleftarrow{\phi^t_d} T^t_d \mbox{,}
    \]
  inductively described as follows:
  Let $t'$ be any descendant of $t$ interpreted as an element $t' \in T^t_i$.
  If $\mathtt{Next}(t') = \mathtt{fail}$ then $t'$ has no (further) preimages under $\phi_{i+1}$.
  Otherwise each evaluation $\mathtt{Next}(t') \in T^t_{i+1}$, which is a preimage of $t'$ under $\phi^t_{i+1}$.
  We call $T^t_\bullet$ the \textbf{tree associated to $t$}.
  Its set of leaves $T^t \coloneqq \lim T^t_\bullet$ is the union of $T(t) \subseteq T$ and the set of all tree-iterators which are descendants of $t$ but have no own descendants.
\end{rmrk}

\begin{rmrk} \label{rmrk:B3}
  Given a tree-iterator $t$ within $T$ with corresponding tree $T^t_\bullet$ as in \Cref{rmrk:B2} we define the subtree
  \[
      T_\bullet: T_0 \xleftarrow{\phi_1} T_1 \xleftarrow{\phi_2} T_2 \xleftarrow{\phi_3} \cdots \xleftarrow{\phi_d} T_d \mbox{.}
    \]
  with $\lim T_\bullet = T(t) \subseteq T$, i.e., the subtree $T_\bullet \subseteq T^t_\bullet$ consisting of the leaves in $T(t)$ and all their predecessors.
  We call $T_\bullet$ the \textbf{tree of relevant leaves of $t$}.
  
  In order to iterate over a tree $T_\bullet$ with set of leaves $T = \lim T_\bullet$ we use the freedom to construct a tree-iterator $t$ within $T$ having $T_\bullet$ as its tree of relevant leaves, even though its associated tree $T^t_\bullet$ might be considerably larger.
\end{rmrk}

\section{Parallel Evaluation of Recursive Iterators} \label{sec:parallel}

In this Appendix we describe the three algorithms
\begin{itemize}
  \item \textbf{ParallellyEvaluateRecursiveIterator} (Algorithm~\ref{algo:peri}),
  \item \textbf{EvaluateRecursiveIterator} (Algorithm~\ref{algo:eri}),
  \item \textbf{LeafIterator} (Algorithm~\ref{algo:leaf-iterator}),
\end{itemize}
which constitute our general parallelization scheme for recursive iterators.
They are independent of any specific recursive iterator (e.g., the one defined by \textbf{IteratorFromState} in Algorithm~\ref{algo:state}).
Furthermore, the recursive iterators can be implemented in classical sequential code, i.e., this organization requires no pre-knowledge in parallel programming in order to implement a recursive iterator and evaluate it in parallel.
We have implemented the three algorithms in the High-Performance-Computing (HPC) version\footnote{Since version 4.9.1 \GAP can be compiled with the option \verb+--enable-hpcgap+.} of \GAP 4.9.2 \cite{GAP492} as part of the \GAP-package $\mathtt{ParallelizedIterators}$ \cite{ParallelizedIterators}.

The combination of these three algorithms takes a recursive iterator $t$ (within $T$) as input and returns an iterator $\ell(t)$ which iterates over the set of leaves $T(t) \subseteq T$.
We call $\ell(t)$ the \textbf{leaf-iterator} associated to $t$.
If $t$ is a tree-iterator then $\ell(t)$ produces no duplicates.

We now briefly explain the role of each of the three algorithms and the way they interact:
Algorithm~\ref{algo:leaf-iterator} is executed in the main thread with a recursive iterator as input.
In the main application of this paper the input is the tree-iterator $t^{(m_k)}$ of all rank $3$ matroids of a given multiplicity vector $(m_k)$, constructed using Algorithm~\ref{algo:state}.
Algorithm~\ref{algo:leaf-iterator} then initializes a global FIFO $L$ of leaves and invokes Algorithm~\ref{algo:peri}.
The latter creates a shared priority queue $P$, launches as many workers (threads) as specified by the user, triggers Algorithm~\ref{algo:eri} in each of them, and then terminates.

The shared\footnote{Implementations of priority queues exist both for shared memory and distributed operation. We have chosen to use a simple shared memory implementation, as contention for our workloads is very low, so we do not have to worry about the priority queue becoming a serialization bottleneck.
} priority queue
stores the list of recursive iterators still to be searched along with their priority, which in our case is the depth at which they were created.
The instance of Algorithm~\ref{algo:eri} running in each thread asks for the highest priority iterator $t'$ in the priority queue $P$ and evaluates $t'' \coloneqq \mathtt{Next}(t')$.
If $t''$ is an element of $T$ then $t''$ is added to the FIFO $L$ of leaves and $t'$ is returned to $P$ with the same priority.
If $t''$ is again an iterator then $t'$ and $t''$ are returned to $P$; $t'$ is returned with the same priority and $t''$ with the priority of $t'$ increased by one.
Finally if $t'' = \mathtt{fail}$ then nothing is done.
After any of the three actions the instance of Algorithm~\ref{algo:eri} starts over again.
In particular, our use of a priority queue avoids the need for a central process supervising the workers.

\bigskip
\begin{algorithm}[H]
\SetKwIF{If}{ElseIf}{Else}{if}{then}{elif}{else}{}%
\DontPrintSemicolon
\SetKwProg{ParallellyEvaluateRecursiveIterator}{ParallellyEvaluateRecursiveIterator}{}{}
\LinesNotNumbered
\KwIn{
\begin{itemize}
  \item A recursive iterator $t$
  \item a number $n \in \N_{>0}$ of workers
  \item a global FIFO $L = ()$, accessible by the subprocesses of the workers
\end{itemize}
}
\KwOut{no return value; the side effect is to fill the FIFO $L$ with the leaves in $T(t)$}
\ParallellyEvaluateRecursiveIterator(){($t$, $n$, $L$)}{
 \nl  Initialize a farm $w$ of $n$ workers $w_1,\ldots, w_n$\;
 \nl  Initialize a \emph{shared} priority queue $P$ of iterators and set $P = ()$\;
 \nl  Initialize a \emph{shared} counter $j$ of jobs in process and pending and set $j = 1$\;
 \nl  Initialize a \emph{shared} semaphore $s \geq 0$ and set $s = 0$\;
 \nl  $P := ( (t,0) )$\;
 \nl  \For{$i=1,\ldots, n$}{
 \nl  $\mathtt{EvaluateRecursiveIterator}(n,L,P,s,j)$ within worker $w_i$\;
 }
 \nl  $\mathtt{SignalSemaphore}(s)$\;
 \nl  \Return{none}\;
 }
\caption{ParallellyEvaluateRecursiveIterator \label{algo:peri}}
\end{algorithm}

Algorithm~\ref{algo:peri} gets as input a recursive iterator, a number $n$ of workers, and a global FIFO $L$.
It initializes a shared priority queue $P$, adds $P$ as the only job with priority $0$, triggers $n$ workers (running in threads) each executing Algorithm~\ref{algo:eri}.
If a worker produces a leaf it writes it to the FIFO $L$.

\bigskip
\begin{algorithm}[H]
\SetKwIF{If}{ElseIf}{Else}{if}{then}{elif}{else}{}%
\DontPrintSemicolon
\SetKwProg{EvaluateRecursiveIterator}{EvaluateRecursiveIterator}{}{}
\LinesNotNumbered
\KwIn{
\begin{itemize}
  \item a number $n \in \N_{>0}$ of all workers
  \item a global FIFO $L = ()$, accessible by the other $n-1$ workers
  \item a shared priority queue $P$
  \item a shared semaphore $s$
  \item a shared counter $j$ of jobs in process or pending
\end{itemize}
}
\KwOut{no return value; the side effect is to evaluate the recursive iterators in the priority queue which get processed by this worker and save the leaves in the FIFO $L$}
\EvaluateRecursiveIterator(){($n$, $L$, $P$, $s$, $j$)}{
 \nl  \While{$\mathtt{true}$}{
 \nl  $\mathtt{WaitSemaphore}(s)$ \label{line:sem} \tcp*{wait until the semaphore $s>0$}
 \nl  \If(\tcp*[f]{if the priority queue is empty}){$P = ()$}{
 \nl  \Return{none} \label{line:return_none} \tcp*{terminate the worker}
 }
 \nl  $(t_i,p_{t_i}) :=\mathtt{Pop}(P)$ \tcp*{get the highest priority job from $P$}
 \nl  $r_i \coloneqq \mathtt{Next}(t_i)$ \label{line:next} \tcp*{pop the recursive iterator $t_i$}
 \nl  \uIf(\tcp*[f]{the result $r_i$ is a leaf}){$r_i \in T$}{
 \nl  $\mathtt{Add}(L,r_i)$ \label{line:AddToFIFO} \tcp*{add the leaf $r_i$ to the FIFO $L$ of leaves}
 \nl  $\mathtt{Add}(P, (t_i,p_{t_i}))$ \label{line:Add} \tcp*{return the recursive iterator $t_i$ back to $P$}
 }
 \nl  \uElseIf(\tcp*[f]{the result $r_i$ is a recursive iterator}){$r_i \neq \mathtt{fail}$}{
 \nl  $\mathtt{Add}(P, (t_i,p_{t_i}))$ \tcp*{return the recursive iterator $t_i$ back to $P$}
 \nl  $\mathtt{SignalSemaphore}(s)$ \tcp*{increase the semaphore by 1}
 \nl  $\mathtt{Add}(P, (r_i,p_{t_i}+1))$ \tcp*{add the new recursive operator $r_i$ to $P$}
 \nl  $\mathtt{SignalSemaphore}(s)$ \tcp*{increase the semaphore by 1}
 \nl  $j := j + 1$ \tcp*{increase the job counter $j$ by 1}
 }
 \nl  \Else(\tcp*[f]{the result $r_i$ is $\mathtt{fail}$}){
 \nl  $j := j - 1$ \tcp*{decrease the job counter $j$ by 1}
 }
 \nl  \If(\tcp*[f]{no recursive iterator is in process or pending}){$j=0$}{
 \nl  $\mathtt{Add}(L, \mathtt{fail})$ \tcp*{add $\mathtt{fail}$ to the FIFO $L$ of leaves}
 \nl  \For(\tcp*[f]{for each worker}){$i=1,\ldots, n$}{
 \nl  $\mathtt{SignalSemaphore}(s)$ \tcp*{increase the semaphore by 1}
 \tcc{the first worker who realizes that there are no jobs left writes $\mathtt{fail}$ in the FIFO $L$ of leaves and increases the semaphore by $n$ to enable all workers to bypass line~\ref{line:sem}, reach line~\ref{line:return_none} and terminate}
 }
 }
 }
 }
\caption{EvaluateRecursiveIterator \label{algo:eri}}
\end{algorithm}

Algorithm~\ref{algo:eri} is the one executed by each worker.
It gets the global state consisting of the number $n$ of workers, the FIFO $L$ of leaves, the priority queue $P$, the semaphore $s$, and the counter $j$ of jobs in process or pending.
A semaphore is a globally shared variable with nonnegative integers as admissible values, which we use to tell workers when to start looking for jobs to process.
The command $\mathtt{SignalSemaphore}(s)$ increases $s$ by $1$.
The command $\mathtt{WaitSemaphore}(s)$ halts until $s>0$ and then decreases $s$ by $1$.

Algorithm~\ref{algo:eri} could be refined for \emph{locally uniform} recursive iterators as follows: Whenever a recursive iterator starts to evaluate leaves then do not add it back to the priority queue (line~\ref{line:Add}) but evaluate it fully (by repeating lines~\ref{line:next} and \ref{line:AddToFIFO}).

In Algorithms~\ref{algo:peri} and \ref{algo:eri} the FIFO $L$ can be equipped with a capacity $k$.
Once this capacity is reached line~\ref{line:AddToFIFO} of Algorithm~\ref{algo:eri} will automatically pause the worker until some other process, e.g. Algorithm~\ref{algo:leaf-iterator}, pops the FIFO $L$.

Algorithm~\ref{algo:leaf-iterator} turns a recursive iterator $t$ within $T$ into a single iterator $\ell(t)$ which enumerates $T(t) \subseteq T$.

\bigskip
\begin{algorithm}[H]
\SetKwIF{If}{ElseIf}{Else}{if}{then}{elif}{else}{}%
\DontPrintSemicolon
\SetKwProg{LeafIterator}{LeafIterator}{}{}
\LinesNotNumbered
\KwIn{
\begin{itemize}
  \item A recursive iterator $t$
  \item a number $n \in \N_{>0}$  of workers
\end{itemize}
}
\KwOut{The associated leaf-iterator $\ell(t)$}
\LeafIterator(){($t$, $n$, $k$)}{
 \nl Initialize a FIFO $L := ()$\;
 \nl Trigger $\mathbf{ParallellyEvaluateRecursiveIterator}(t,n,L)$\;
 \nl Initialize the leaf-iterator $\ell$:\;
 \nl \phantom{aa} Define $\mathtt{IsDone}(\ell)$ to check if first entry of $L$ is $\mathtt{fail}$\footnote{Recall, $\mathtt{fail} \notin T$.}\;
 \nl \phantom{aa} Define $\mathtt{Next}(\ell)$ to return the first entry of $L$ which is an element of $T(t)$\;
 \nl \Return{$\ell$}\;
 }
\caption{LeafIterator (Leaf-iterator of a recursive iterator) \label{algo:leaf-iterator}}
\end{algorithm}

\section{Why \HPCGAP?} \label{sec:WhyHPCGAP}

We list some advantages of our implementation in \HPCGAP:
\begin{enumerate}
  \item More threads can be added on the fly; they simply start to pull jobs from the priority queue (if nonempty);
  \item One can even notify single threads to terminate once they finish evaluating a recursive-iterator;
  \item \HPCGAP supports global shared memory and therefore allows us to use a simple and efficient shared memory implementation for priority queues, as described in Section~\ref{sec:iterators};
  \item \HPCGAP allows for objects to be moved efficiently from one thread to another by reassigning ownership of those objects to the new thread, rather than inefficiently performing a full structural copy or using serialization.
\end{enumerate}

The most obvious drawback of our implementation is the following:
The state of evaluation of a recursive iterator is defined by the priority queue (residing in a shared region) and by the iterators that are being evaluated in the threads.
  So if a thread dies or hangs\footnote{either manually terminated or due to an instability of \HPCGAP, which rarely happens in the current version} while evaluating a recursive-iterator then the latter (which was adopted by the thread from the priority queue) with all its leaves (e.g., matroids) are lost.
  In particular, it is impossible to terminate the running \HPCGAP process without losing the state of evaluation.
  
  A second drawback is that it is currently impossible to use a distributed computational model since in our implementation the state of evaluation of a recursive iterator can only be defined and managed by a single \HPCGAP process.

One way to avoid these drawbacks is to store the state of evaluation into a (temporary) database.
In particular \emph{all} yet nonfully evaluated recursive-iterators should be stored in the database, while those in process should be marked as such using a unique fingerprint of the evaluating process.
This allows a distributed access on the one side.
On the other side an iterator with a deadlock can be manually (or maybe even automatically by a watchdog) be freed for evaluation by other threads searching for jobs.

Our implementation performs best for recursive-iterators where the evaluation time of each produced iterator is considerably longer than the organizational overhead in  \HPCGAP caused by redefining regions, etc.

\section{Timings} \label{sec:timings}

It is worth noting that $97\%$ of the $404$ tree-iterators of the different multiplicity vectors for $n=13$ atoms can be evaluated in less than a day of CPU time.
For $n=14$ the corresponding number are still $93\%$ of $695$.

\begin{rmrk}
While processing all relevant multiplicity vectors is an ``embarrassingly parallel'' problem, the reader may have noticed that the parallel evaluation of a single tree-iterator corresponding to one such multiplicity vector is much more involved.
\end{rmrk}

The gain of the parallelized evaluation of tree-iterators of rank $3$ matroids with given multiplicity vector depends on the number $n$ of atoms.
The longest CPU time of an evaluation of a tree-iterator with $n=13$ atoms was that of the one with multiplicity vector $(m_3,m_4) = (18,4)$ which took $16.2$ CPU days but finished in $5.59$ days using $8$ workers, a factor of $2.9$.
The gain for $n=14$ was more significant:
The multiplicity vector with the largest number of matroids is $(m_2,m_3,m_4,m_5) = (14,9,5,2)$.
It generated $168352$ matroids ($45$ of them are representable) in about $22.8$ hours of CPU time, but finished in $112$ minutes using 24 workers, a factor of $12.2$.
The multiplicity vector with the longest CPU time for evaluating the tree-iterator is $(m_2,m_3,m_4,m_5) = (3,18,4,1)$.
It generated $34$ matroids (only one of them is representable) and took $495.7$ CPU days but finished in $74.3$ days using $8$ workers, a factor of $6.7$.

\bibliographystyle{myalpha}
\newcommand{\includebibliography}[1]{\bibliography{#1/a,#1/b,#1/c,#1/d,#1/e,#1/f,#1/g,#1/h,#1/i,#1/j,#1/k,#1/l,#1/m,#1/n,#1/o,#1/p,#1/q,#1/r,#1/s,#1/t,#1/u,#1/v,#1/w,#1/x,#1/y,#1/z}}
\def\cprime{$'$} \def\cprime{$'$} \def\cprime{$'$} \def\cprime{$'$}
  \def\cprime{$'$}
\providecommand{\bysame}{\leavevmode\hbox to3em{\hrulefill}\thinspace}
\providecommand{\MR}{\relax\ifhmode\unskip\space\fi MR }
\providecommand{\MRhref}[2]{%
  \href{http://www.ams.org/mathscinet-getitem?mr=#1}{#2}
}
\providecommand{\href}[2]{#2}

\end{document}